\documentclass[graybox]{amsart}

\usepackage{amssymb,a4wide,amsthm,amsmath,amsfonts,epsfig}
\usepackage{mathrsfs}
\input colordvi
\usepackage[active]{srcltx}
\usepackage{comment}
\newcommand{\eps}{\varepsilon}

\theoremstyle{plain}
\newtheorem{theorem}{Theorem}[section]
\theoremstyle{remark}
\newtheorem{remark}[theorem]{Remark}

\theoremstyle{plain}
\newtheorem{corollary}[theorem]{Corollary}
\newtheorem{lemma}[theorem]{Lemma}
\newtheorem{proposition}[theorem]{Proposition}

\numberwithin{equation}{section}
\renewcommand\div{\mathrm{div}\,}

\begin{document}
\title[SDEs on Riemannian manifolds]{Attractivity, invariance and ergodicity for SDEs on Riemannian manifolds}%
\author[L. Ba\v nas, Z. Brze\'zniak, M. Ondrej\'at \& A. Prohl]{L. Ba\v nas, Z. Brze\'zniak, M. Ondrej\'at \& A. Prohl}
\address{Department of Mathematics, Heriot-Watt University,
                EH14 4AS Edinburgh, United Kingdom}
\email{l.banas@hw.ac.uk}
\address{Department of Mathematics, The University of York,
         Heslington, York YO10 5DD, United Kingdom}
\email{zb500@york.ac.uk}
\address{Institute of Information Theory and Automation, Pod Vod\'arenskou v\v e\v z\'{\i} 4, CZ-182 08, Praha 8, Czech Republic, phone: ++ 420 266 052 284, fax: ++ 420 286 890 378}
\email{ondrejat@utia.cas.cz}
\address{Mathematisches Institut, Universit\"at T\"ubingen, Auf der Morgenstelle 10, D-72076 T\"ubingen, Germany}
\email{prohl@na.uni-tuebingen.de}
\thanks{Full version of the paper with figures is available from {\tt http://na.uni-tuebingen.de/preprints.shtml}}
\thanks{The research of the third named author was partially supported by the GA \v CR Grant P201/10/0752.}
\maketitle
\begin{abstract}
We give a sufficient condition on nonlinearities of an SDE on a compact connected Riemannian manifold $M$ which implies that laws of all solutions converge weakly to the normalized Riemannian volume measure on $M$. This result is further applied to characterize invariant and ergodic measures for various SDEs on manifolds.
\end{abstract}

\section{Introduction}

This work has its origins in an attempt (so far not concluded) by the authors to find a  mathematically sound numerical approximation for the stochastic geometric wave equations. Such equations have been recently introduced by two of us in \cite{Brz+O_2007}, and our aim is to produce a counterpart of the numerical scheme introduced by three of us in \cite{bbp1} for the stochastic Landau-Lifshitz-Gilbert equations (used in the theory of ferromagnetism). For the deterministic geometric wave equations such questions have recently been studied by two of us and Sch\"{a}tzle in \cite{bpsch}. During the initial stages we encountered problems related to large time behaviour of the solutions and we realised that we did not know an answer to such problems even for finite dimensional models. The reason being that the finite dimensional  approximations to stochastic geometric wave equations, for instance those generated by finite elements methods,  are highly degenerate with respect to noise. Moreover, there are not many papers on the stochastic Langevin equations  on manifolds, but see  \cite{jorgensen_1978} and \cite{Baudoin+H+T_2008}.

On the other hand, the existence and uniqueness of invariant measures or convergence of solutions to an attracting measure was studied in case of uniformly elliptic diffusion operators e.g. in \cite{ikwa} and in case of hypoelliptic parabolic and elliptic systems in \cite{ich1} and \cite{ich2}. The variety of interesting problems is of course much wider than the cases described above as even quite simple SDE problems degenerate in such a way that none of the existing results covers them (see the examples at the end of this paper). A particular approach to such degenerated elliptic problems is always required and we present one of them - still quite general - which implies that laws of all solutions converge weakly to the normalized Riemannian volume measure on $M$. For instance, in Theorem \ref{thm-2} we formulate sufficient conditions for the existence and uniqueness of an attractive measure.  We develop further this result to characterize invariant and ergodic measures for various degenerated SDEs on manifolds. To be more precise, in Theorems \ref{thmabneq0} and \ref{n3} we give a complete characterization of the set of invariant and invariant ergodic  measures for  certain classes of SDEs on the sphere $\Bbb S^2$, and the   special orthogonal group $SO_3$.

A further goal of this paper is to construct numerical schemes for solving certain classes of SDEs on manifolds, i.e., the finite dimensional stochastic Landau-Lifshitz-Gilbert (LLG) equation, and the geodesic equation on the
sphere with stochastic forcing. The LLG equation has been widely studied in the physical literature, see for instance \cite{Chantrell_2009} and references therein. They have also received some attention from the mathematical point of view, see \cite{Kohn_R_VdE_2005}. General stochastic LLG's for non-uniform magnetisations have been studied by one of us and Goldys in \cite{Brz+G_2009}, while  Leli{\`e}vre et al in \cite{Lelivre+LeBris_2008} have described certain numerical schemes for SDEs with constraints. A convergent discretisation in space and time which is based on finite elements is proposed in \cite{bcp1}; this scheme guarantees the sphere constraint to hold for approximate magnetisation processes, and thus inherits the Lyapunov structure of the problem. As a consequence, iterates may be shown to construct weak martingale solutions of the limiting equations. Main steps of this construction are detailed for the finite dimensional LLG equation in Section~\ref{nappr1} to then study long-time dynamics in Section~\ref{nexp1}. A corresponding numerical program for second order equations with stochastic forcing is detailed in Section~\ref{napprox2},
which requires different tools;
in particular, a discrete Lagrange multiplier is used in Algorithm B for iterates to inherit the sphere constraint
in a discrete setting. Overall convergence of iterates is asserted in Theorem~\ref{numscheme1} which holds
for this particular SDE on the sphere, but which may also be considered as a first step to numerically approximate the stochastic geometric wave equation. Again, computational examples are provided in Section~\ref{nelubo}
to illustrate the
results proved in this work, and motivate further analytical studies for computationally observed long-time
behaviors which lack a sound analytical understanding at this stage.

\section{The problem and preliminary results}

Let $M$ be a compact connected $d$-dimensional Riemannian manifold whose normalized Riemannian volume measure is denoted by $\lambda$ (i.e. $\lambda(M)=1$). Integration with respect to $\lambda$ will be denoted by  $dx$. Let us assume that $F$ and  $G_1,\dots,G_m$ are  smooth vector fields on $M$. Let $W^1,\dots,W^m$ be independent $(\mathscr F_t)$-Wiener processes on a complete stochastic basis $(\Omega,\mathscr F,(\mathscr F_t),\mathbb{P})$ and consider the following Stratonovich SDE
\begin{equation}\label{sde}
du=F(u)\,dt+\sum_{k=1}^m\, G_k(u)\circ\,dW^k
\end{equation}
on $M$.

The following (preliminary) results on SDE on manifolds in this section are generally known and can be found e.g. in \cite{elw}, \cite{ikwa} or \cite{kun}. Denote by $P=\big(P_t\big)_{t\geq 0}$ the Feller semigroup associated with the SDE \eqref{sde}, i.e.
$$P_th(x)=\mathbb{E}\big(h(u^x(t))\big),\;\; h\in B_b(M),\;\; (t,x)\in\mathbb{R}_+\times M$$
 by $B_b(M)$ we denote the Banach space of all real bounded and Borel measurable functions on  $M$, and $u^x$ is a solution of (\ref{sde}) such that $$u^x(0)=x.$$ Let us also denote by $C(M)$ the separable Banach space of all real continuous  functions on  $M$, let us  recall that the restriction of $P$ to $C(M)$  is a $C_0$-semigroup on $C(M)$, denote by  $\mathcal A$ its infinitesimal generator and recall the following characterization:
 \begin{eqnarray}
\nonumber
C^2(M)&\subseteq& D(\mathcal A)\\
\mathcal Af&=&Ff+\frac12\sum_{k=1}^mG_k(G_kf),\qquad f\in C^2(M).
\label{generator}
\end{eqnarray}

We will denote by $P^\ast$ the dual semigroup on the space $\mathcal{M}(M)$ of all $\mathbb{R}$-valued Borel measures on $M$, and by ${\mathcal A}^\ast$ the infinitesimal generator of the dual semigroup.

We convene that, throughout this paper, all measures on $M$ will be Borel, i.e. with the domain $\mathscr B(M)$.

Let us  also introduce the transition kernel

 \begin{equation}\label{eqn-203}
 p(t,x,A)=\mathbb{P}\,\big(\{u^x(t)\in A\}\big),\;;\  t\ge 0, \; x\in M, \; A\in\mathscr B(M).\end{equation}

Let $\nu$ be a probability measure on $M$. We say that $\nu$ is {\it invariant}\footnote{We will also say that $\nu$ is invariant for the SDE \eqref{sde} that generates the semigroup $P$.} provided
\begin{equation}\label{eqn-204}
\int_Mp(t,x,A)\,d\nu=\nu(A), \qquad  \mbox{ for all }t\ge 0,\quad x\in M \mbox{ and }\; A\in\mathscr B(M).
\end{equation}
An invariant probability measure is said to be {\it ergodic} if it is an extremal point of the convex set of all invariant probability measures on $M$.

A probability measure $\nu$ on $M$ is called {\it attractive} if
\begin{equation}\label{eqn-205}
\lim_{t\to\infty}P_tf(x)=\int_Mf\,d\nu,\qquad \mbox{ for all } x\in M \mbox{ and }\; f\in C(M)
\end{equation}
which, in other words, means that every solution $u$ of (\ref{sde}) converges in law to $\nu$ irrespectively of the initial condition (even if $u(0)$ is random).

Let $L_0=\{F,G_1,\dots,G_m\}$ and $L_n=\{[X,Y]:X,Y\in\bigcup_{j=0}^{n-1}L_j\}$ for $n\in\mathbb{N}$, and denote
\begin{equation}\label{lieal}
\mathscr L=\operatorname{span}\bigcup_{n=0}^\infty L_n,\qquad\mathscr L_F=\operatorname{span}\Big[(L_0\setminus\{F\})\cup\bigcup_{n=1}^\infty L_n\Big]
\end{equation}
where $[X,Y]=XY-YX$. Then we say that the vector fields $F,G_1,\dots,G_m$  satisfy the hypothesis
\begin{itemize}
\item[$(\mathbf{H})$] if $\operatorname{dim}\{X_p:X\in\mathscr L\}=d$ for every $p\in M$ (the H\"ormander condition),
\item[$(\mathbf{F})$] if every $f\in C^\infty(M)$ is constant on $M$ provided $Lf=0$ on $M$ for every $L\in\mathscr L_F$,
\item[$(\mathbf{D})$] if $\operatorname{div}G_1=\dots=\operatorname{div}G_m=\operatorname{div}F=0$ on $M$,
\item[$(\mathbf{C})$] if, for every $l\in\mathbb{N}$, there exists a finite dimensional subspace $C_l$ of $C^2(M)$ such that $\mathcal A(C_l)\subseteq C_l$  and $\operatorname{span}\bigcup_{l\in\mathbb{N}} C_l$ is dense in $C(M)$  with respect to  the supremal norm.
\end{itemize}

\begin{remark} Let us observe that the equation (\ref{sde}) and the hypotheses $(\mathbf{H})$, $(\mathbf{F})$ and $(\mathbf{C})$ do not depend on the Riemannian structure of the manifold $M$. In particular,  we can, if necessary (and possible), change the metric on $M$ in order the vector fields $F,G_1,\dots,G_m$ satisfy the hypothesis $(\mathbf{D})$. Of course, a change of the metric on $M$ changes accordingly the Riemannian volume measure on $M$.
\end{remark}

\begin{remark} The hypothesis $(\mathbf{F})$ is clearly satisfied if either the union of closures of the components of $\{p\in M:\operatorname{dim}\{X_p:X\in\mathscr L_F\}=d\}$ is a connected dense subset of $M$, or if any two points of $M$ can be connected by a piecewise smooth curve consisting of integral curves of the vector fields in $\mathscr L_F$ (see the examples in Section \ref{ExamI} and \ref{ExamII}).
\end{remark}

\begin{proposition}\label{prop-charinv} A  probability measure $\nu$ on $M$ is invariant if and only if
$$
\int_M\mathcal Ah\,d\nu=0,\qquad h\in C^2(M).
$$
\end{proposition}
\begin{proof} See the identity (4.58) on p. 292 in \cite{ikwa}.
\end{proof}

\begin{proposition}[Krylov-Bogolyubov] There exists at least one invariant measure.
\end{proposition}

\begin{proof} See Corollary 3.1.2 in \cite{daza}.
\end{proof}

\begin{proposition}\label{prop-thdiv} Let $R\in C^\infty(M)$ be a non-negative function. Then the measure $d\mu=R\,d\lambda$ is an invariant measure iff the following equality holds on $M$
\begin{equation}\label{dive}
\frac 12\sum_{k=1}^d\operatorname{div}\{[\operatorname{div}(RG_k)]G_k\}=\operatorname{div}(RF).
\end{equation}
\end{proposition}

\begin{proof}
This result can be proved directly as for instance the last but one claim of Theorem 6.3.2 from \cite{D+S_1989} or it can be deduced from that result.
\end{proof}

\begin{corollary}\label{cor-divcondi} If $(\mathbf{D})$ holds then
\begin{equation}\label{generatoradj}
\mathcal A^\ast f=-Ff+\frac12\sum_{k=1}^mG_k^2(f),\qquad f\in C^2(M).
\end{equation}
Moreover,  the condition (\ref{dive}) is equivalent to $\mathcal A^\ast R=0$.
 \end{corollary}
\begin{proof}
Since for a vector field $X$ on $M$ and any $f\in C^2(M)$, we have
\begin{eqnarray*}X^\ast f &=&-\div (f X)=-f \div X  -X(f), \\
(X^2)^\ast  f&=&f(\div X)^2+2X(f)\div X+ fX(\div X)+X^2(f),
\end{eqnarray*}
by Proposition \ref{prop-thdiv} we get the result.
\end{proof}

\begin{proposition}\label{prop-horem} If $(\mathbf{H})$ holds then
every invariant measure has a density $R\in C^\infty(M)$ with respect to the Riemannian volume measure $\lambda$ on $M$.
\end{proposition}

\begin{proof}
This follows from \cite{horm}.
\end{proof}

The following is a consequence of Theorem 6.3.2 from \cite{D+S_1989} (or of our Proposition \ref{prop-thdiv}).
\begin{theorem}\label{thm-2.8} The normalized Riemannian volume measure $\lambda$ on $M$ is invariant  iff
\begin{equation}
\label{eqn-nsc-inv}
\div(F)+\sum_{k=1}^m\big[ G_k(\div G_k)-(\div G_k)^2]=0 \; \mbox{ on } M.
\end{equation}
In particular, if the condition  $(\mathbf{D})$ is satisfied,  then $\lambda$  is invariant. Moreover, if both conditions  $(\mathbf{D})$ and $(\mathbf{H})$ are satisfied, then $\lambda$ is the unique invariant probability measure.
\end{theorem}

\begin{proof} It is sufficient to prove the last claim. So let us assume that  $\mu$ is an invariant probability measure.   Then, in view of Proposition \ref{prop-horem},  $\mu$ has a density $R$ with respect to $\lambda$ and $R\in C^\infty(M)$. Moreover, by Corollary \ref{cor-divcondi}, $\mathcal A^\ast R=0$. On the other hand, in view of the first formula in the proof of Corollary \ref{cor-divcondi}, from the assumption $(\mathbf{D})$ we infer that
\begin{equation}\label{eqn-R}
\int_MR\cdot\mathcal A^\ast R\,dx=-\frac12\sum_{k=1}^m\int_M|G_kR|^2\,dx.
\end{equation}
Hence $G_1R=\dots=G_mR=0$ and consequently $FR=-\mathcal A^\ast R=0$. Thus $$LR=0 \mbox{ for every }L\in\mathscr L(F,G_1,\dots,G_m).$$
The condition $(\mathbf{H})$ yields that $R$ is a constant function.
\end{proof}

\begin{theorem}\label{thm-2} Asuume the that the Hypothesis $(\mathbf{F})$, $(\mathbf{D})$ and $(\mathbf{C})$ hold. Assume that a process  $u$ is  a solution to SDE  (\ref{sde}). Then there exists a probability measure $\theta$ on $M$ such that
\begin{equation}\label{eqn-lim-u}
\lim_{t\to\infty}\mathbb{E}\,f(u(t))=\int_Mf\,d\theta,\qquad f\in C(M).
\end{equation}
If, in addition, $(\mathbf{H})$ holds then $\theta=\lambda$.
\end{theorem}

\begin{proof} Let us fix a natural number  $l\in\mathbb{N}$.  Let us choose  a basis $\{f_1,\dots,f_n\}$   of $C_l$ and let $(a_{ij})_{i,j=1}^n$ be a a real $(n\times n)$-matrix of the restriction of linear operator $\mathcal A$ to $C_l$ with respect to this basis. In particular,
$$\mathcal Af_i(p)=\sum_{j=1}^na_{ij}f_j(p), \;\;p\in M,\;\; i\in\{1,\dots,n\}.
$$
Let us denote by $A$ the linear operator in $\mathbb{R}^n$ whose matrix in the canonical basis $\{e_1,\cdots,e_n\}$ is equal to  $(a_{ij})_{i,j=1}^n$.

Let  a process $u$ be  a solution to SDE (\ref{sde}).
 Then  for every $i\in\{1,\dots,n\}$
\begin{eqnarray*}
\mathbb{E}\,f_i(u(t))&=&\mathbb{E}\,f_i(u(0))+\mathbb{E}\int_0^t(\mathcal Af_i)(u(s))\,ds
\\
&=&\mathbb{E}\,f_i(u(0))+\sum_{j=1}^na_{ij}\int_0^t\mathbb{E}\,f_j(u(s)\,ds,\qquad t\ge 0.
\end{eqnarray*}
  Hence
\begin{equation}\label{bdd}
\sum_{i=1}^n \mathbb{E}\,f_i(u(t)) e_i=e^{At} \big[ \sum_{j=1}^n \mathbb{E}\,f_j(u(0)) \, e_j\big],\qquad t\ge 0.
\end{equation}
Since by  Lemma \ref{lem-convexp} below the operator valued  function $e^{At}$ is convergent as $t\to\infty$, in view of \eqref{bdd} we infer that so are the functions $\mathbb{E}\,f_i(u(t))$.
By the  density part of assumption $(\mathbf{C})$, we conclude that for every $f\in C(M)$, $\mathbb{E}\,f(u(t))$ is convergent as $t\to\infty$.

Now we will prove the last claim and so we assume that the condition $(\mathbf{H})$ holds.
Then by the Krylov-Bogolyubov Theorem in conjunction with Theorem \ref{thm-2.8},
$$
\lim_{t\to\infty}\frac 1t\int_0^t\nu_s(B)\,ds=\lambda(B),\qquad B\in\mathscr B(M),
$$
where $\nu_t$ is the law of $u(t)$. This concludes  the proof.
\end{proof}

\begin{lemma}\label{lem-convexp} The operator valued  function $e^{tA}$ from the proof of Theorem \ref{thm-2} is convergent as $t\to\infty$.
\end{lemma}

\begin{proof}[Proof of Lemma \ref{lem-convexp}] Let us begin by introducing the following notation:
$$y(p)=(f_j(p))_{j=1}^n =\sum_{j=1}^n f_j(p)e_j \in \mathbb{R}^n, \;\;p\in M.$$
If follows from  equality (\ref{bdd}) that for any $p\in M$,  $\sup_{t\ge 0}|e^{At}y(p)|<\infty$. On the other hand,
since the functions $f_1,\dots,f_n$ are linearly independent in  $C^2(M)$,   the vectors $\{y(p):p\in M\}$ span $\mathbb{R}^n$. Hence, we infer
 that
\begin{equation}\label{maxa}
\sup_{t\ge 0}\Vert e^{At}\Vert <\infty.
\end{equation}
In the above we considered $A$ and $e^{At}$ as linear operators on $\mathbb{R}^n$. We may naturally extend them to the complex space
$\mathbb{C}^n$. Since the euclidean norms of $e^{At}$  in $\mathbb{R}^n$ and $\mathbb{C}^n$ coincide, we infer that
\begin{equation}\label{maxa-C}
\sup_{t\ge 0}\Vert e^{At}\Vert_{L(\mathbb{C}^n)} <\infty.
\end{equation}
It follows from \eqref{maxa-C} that the spectrum $\sigma(A)$ is contained in $\mathbb{C}_-$, the closed real-negative halfplane of $\mathbb{C}$.

We will show that $(i\mathbb{R}\setminus\{0\})\cap \sigma(A)=\emptyset$. Arguing by contradiction let us consider
$\lambda \in (i\mathbb{R}\setminus\{0\})\cap \sigma(A)$. Then
$\lambda=i\alpha$ for some $\alpha\in\mathbb{R}\setminus\{0\}$  belongs to the spectrum of the adjoint operator $A^\ast$. Hence  and there is $z=u+iv\in\mathbb{C}^n\setminus\{0\}$ such that $A^\ast z=i\alpha z$, i.e. $A^\ast u=-\alpha v$ and $A^\ast v=\alpha u$. If we denote $g=\sum_{j=1}^nu_jf_j$ and $h=\sum_{j=1}^nv_jf_j$,  then $\mathcal A^\ast g=-\alpha h$ and $\mathcal A^\ast h=\alpha g$. Hence, by identity \eqref{eqn-R}, we get
$$
0=\int_Mg\big[ \mathcal A^\ast g+h\mathcal A^\ast h\big]\,dx=-\frac12\sum_{j=1}^m\int_M \big[ |G_jg|^2+|G_jh|^2\big]\,dx.
$$

Therefore  $G_jg=G_jh=0$ on $M$ for every $j\in\{1,\dots,m\}$. Consequently, $Fg=-\alpha h$ and $Fh=\alpha g$ and so

\begin{eqnarray*}
[G_j,F]g &=& G_jFg=-\alpha G_jh=0,
\end{eqnarray*}
\begin{eqnarray*}
[G_j,F]h &=& G_jFh=\alpha G_jg=0.
\end{eqnarray*}

By an induction,  we infer that  $Xg=Xh=0$ for every $X\in\mathscr L_F$. Hence,  by assumption $(\mathbf{F})$, the functions $g$ and $h$ are constant  on $M$. Thus $h=-\alpha^{-1}Fg=0$ and $g=\alpha^{-1}Fh=0$ what implies that  $z=0$. This contradiction concludes the proof of our claim.

 Let us  now assume that $0\in\sigma(A)$.  Then there exists a natural number $N\in\mathbb{N}$, matrices $A_k$, $k\in\{1,\dots,N\}$, a positive number $\eps>0$  and a holomorphic $\mathcal{L}(\mathbb{R}^n)$-valued function $V$ defined on an open disc in $\mathbb{C}$ of radius $\eps$ such that
$$
(\lambda-A)^{-1}=\sum_{k=1}^N\lambda^{-k}A_k+V(\lambda),\qquad 0<|\lambda|<\varepsilon.
$$
By the residuum theorem,
$$
e^{tA}=\sum_{k=1}^N\frac{t^{k-1}}{(k-1)!}A_k+o(t)
$$
where $o(t)\to 0$ as $t\to\infty$. We however already know that $e^{tA}$ is bounded and hence, necessarily, $N=1$. Thus we infer that   $e^{tA}\to A_1$ as $t\to\infty$ what concludes the proof.

\end{proof}

\section{Example I}\label{ExamI}
Assume that $A,B_1,\dots,B_m$ are  antisymmetric $(n\times n)$-matrices  and  $W_1,\dots,W_m$ are independent $(\mathscr F_t)$-Wiener processes.
Consider the following stochastic differential equation in the  Stratonovich form,
\begin{equation}\label{eer}
dz=Az\,dt+\sum_{k=1}^mB_kz\circ dW_k
\end{equation}
on the sphere $\mathbb{S}^{n-1}\subseteq\mathbb{R}^n$. Denoting by $\mathcal{L}(\mathbb{R}^n)$ the space of real $(n\times n)$-matrices, every solution of (\ref{eer}) can be written as $Z(t)z_0$,  where $Z$  is an $\mathcal{L}(\mathbb{R}^n)$-valued solution of the following Stratonovich stochastic differential equation
\begin{equation}\label{mateer}
\begin{array}{rcl}
Z(t)&=&AZ\,dt+\sum_{k=1}^mB_kZ\circ\,dW^k, \;\;t\geq 0
\end{array}
\end{equation}
satisfying the initial condition $Z(0)=I$. In fact, the solution $Z$ takes values in the set $SO_n$ called the  special orthogonal group and consisting of all unitary operators in $\mathbb{R}^n$  with determinant $1$. Let us recall that
\begin{itemize}
\item the set $SO_n$ is a compact connected Lie group and a  submanifold of $\mathcal{L}(\mathbb{R}^n)$,
\item for $Z\in SO_n$,
\begin{eqnarray*}T_ZSO_n &=&\{V\in \mathcal{L}(\mathbb{R}^n)
:V+ZV^\ast Z=0\}\\
&=&\{V\in \mathcal{L}(\mathbb{R}^n)
:VZ^\ast+ZV^\ast=0\},
\end{eqnarray*}
\item there exists a bi-invariant Riemannian metric on $SO_n$ and the corresponding  normalized Riemannian volume measure $\lambda$ on $SO_n$ is bi-invariant (with respect to  multiplication by elements of $SO_n$).
\end{itemize}

For an antisymmetric $M\in\mathcal L(\mathbb{R}^n)$, we denote by $\mathbf{M}$ the following vector field on  $SO_n$,
$$\mathbf{M}: SO_n \ni Z\mapsto MZ \in T SO_n.$$

\begin{lemma} In the above framework,  the conditions $(\bf D)$ and $(\bf C)$ are satisfied for the vector fields $\mathbf{A}$ and $\mathbf{B_1}, \cdots, \mathbf{B_k}$.
\end{lemma}

\begin{proof} Let us fix an antisymmetric $M\in \mathcal{L}(\mathbb{R}^n)$. We  have $\operatorname{div}\mathbf{M}=0$. Indeed, if $(E_i)$ are orthonormal vector fields on $SO_n$, then, see \cite{Chavel_2006},
$$
\operatorname{div}\mathbf{M}=\sum_{i=1}^{{n \choose 2}}\langle\nabla_{E_i}\mathbf{M},E_i\rangle_{\mathbb{R}^{n\times n}}.
$$
On the other hand, since $M$ is antisymmetric,
$$
\langle\nabla_{E_i}\mathbf{M},E_i\rangle_{\mathbb{R}^{n\times n}}=\lim_{t\to 0}t^{-1}\langle M(z+tE_iz)-Mz,E_iz\rangle_{\mathbb{R}^{n\times n}}=\langle ME_iz,E_iz\rangle_{\mathbb{R}^{n\times n}}=0.
$$
Combining the above two identities we infer that  the condition $(\bf D)$ is satisfied.

Let us prove the  the hypothesis $(\bf C)$ is also satisfied. Towards this end we identify the space $\mathcal{L}(\mathbb{R}^n)$ with $\mathbb{R}^{n\times n}$.
 Let us denote by $C_l$  the vector space spanned by the restriction to $O_n$ of polynomials on $\mathbb{R}^{n\times n}$ of order smaller or equal to $l$, i.e. $C_l=\operatorname{span}\{f_\alpha:|\alpha|\le l\}$,  where
$$
f_\alpha(X)=X^\alpha=\prod_{j=1}^n\prod_{k=1}^nx_{jk}^{\alpha_{jk}},\qquad X=[x_{ij}]\in\mathbb{R}^{n\times n},\quad\alpha\in\mathbb{N}_0^{n\times n}.
$$
Observe that
$$
(\mathbf{M}f_\alpha)(Z)=\sum_{l,j,k=1}^nm_{jl}\alpha_{jk}z^{\alpha-e^{jk}+e^{lk}}
=\sum_{l,j,k=1}^nm_{jl}\alpha_{jk}f_{\alpha-e^{jk}+e^{lk}}(Z)
, \; Z=[z_{ij}]\in\mathbb{R}^{n\times n}, $$
where $e^{jk}$ is the zero $(n\times n)$-matrix except for the position $(jk)$ where $e^{jk}_{jk}=1$. In particular, since the degree of $f_{\alpha-e^{jk}+e^{lk}}$ is $\leq l$, we infer that  $C_l$ is invariant for the operator $\mathcal A$ defined in (\ref{generator}). Finally,  by the Stone-Weierstrass Theorem, $\bigcup_{l=0}^\infty C_l$ is dense in $C(SO_n)$.
\end{proof}

\begin{corollary} The  normalized Riemannian volume measure $\lambda$ on $SO_n$ is invariant with respect to the Feller semigroup generated by the SDE (\ref{mateer}).
\end{corollary}

\begin{theorem} If the Lie algebra $\mathscr L_A(A,B_1,\dots,B_m)$  contains all antisymmetric $(n\times n)$-matrices, then $\lambda$ is an attractive and  unique invariant probability measure for the Feller semigroup generated by the SDE (\ref{mateer}).
\end{theorem}

As a special case let us assume that $A$ and $B$ are two   antisymmetric $(3\times 3)$-matrices  such that $B\ne 0$. Consider the following Stratonovich SDE on $SO_3$:
\begin{equation}\label{mateer2}
dZ(t)=AZ\,dt+BZ\circ\,dW.
\end{equation}
Since  both $A$ and $B$ are   antisymmetric, we infer that
 \begin{trivlist}
 \item[(i)] $[A,B]=0$ iff $A$ and $B$ are linearly dependent,
   \item[(ii)]
  $\{A,B,[A,B]\}$ spans the $3$-dimensional space of antisymmetric matrices iff $A$ and $B$ are linearly independent
\end{trivlist}
Hence,    two possibilities may arise. We begin with the first one.

\begin{theorem}\label{thmabneq0} If $[A,B]\ne 0$ then the  measure $\lambda$ on $SO_3$ is an attractive and  unique invariant probability measure for for the Feller semigroup generated by the SDE (\ref{mateer2}).
\end{theorem}

If $[A,B]=0$ then  $A$ and  $B$ are commuting  and so $\exp\,\{tA+W_tB\}$ is the unique solution of (\ref{mateer2}). Assume that $B\ne 0$. Let us put $\rho=(\frac 12\|B\|^2_{\mathbb{R}^{3\times 3}})^\frac 12$.
For $K\in\{\mathbb{S}^2,SO_3\}$, define
\begin{eqnarray}\label{eqn-S}
&&S:\mathbb{S}^1\times K \ni (p,Z)\mapsto\mathbf s(p)Z \in  K,\\
\label{eqn-s}
&&\mathbf s(x,y)=\frac{1-x}{\rho^2}B^2+\frac{y}{\rho}B+I, \; (x,y)\in \mathbb{S}^1,
\end{eqnarray}

 Note  that for antisymmetric $(3\times 3)$-matrix $B$,  $B^3=-\rho^2B$. Thus the definition (\ref{eqn-S}-\ref{eqn-s}) of the map  $S$ is  correct. Moreover, since  $\mathbf s(p)\mathbf s(q)=\mathbf s(pq)$ for all $p,q\in\mathbb{S}^1$,
 the map
 \begin{equation}\label{eqn-group-S}
 \mathbb{S}^1\ni p \mapsto S_p= S(p,\cdot)\in K
 \end{equation}
 is a group homomorphism, i.e.  $S(p,S(q,Z))=S(pq,Z)$ for $p,q\in\mathbb{S}^1$ and $Z\in K$. For $Z\in K$ we will denote by $S^Z$ the orbit of the group $\big(S_p\big)_{p\in \mathbb{S}^1}$ through $Z$, i.e. the  function
 \begin{equation}\label{eqn-S^Z}
S^Z: \mathbb{S}^1\ni p \mapsto S(p,Z)\in K
 \end{equation}
 Let also denote by $H$  be the space of all such orbits, i.e.
$$H=\{S^Z[\mathbb{S}^1]:Z\in K\}.$$
The space $H$ is  equipped with the quotient topology for the surjective projection $$\pi:K \ni Z\mapsto S^Z[\mathbb{S}^1]\in  H$$
with respect to which $H$ is compact. The quotient topology of $H$ is metrizable by the classical Hausdorff metric
$$
\rho(X,Y)=\max\,\{\sup_{x\in X}\,d(x,Y),\sup_{y\in Y}\,d(y,X)\}.
$$
Let us denote by  $dx$  the normalized Haar  measure on $\mathbb{S}^1$.
 If $X\in H$ and $U,V\in K$ are two elements belonging  to the orbit of $X$,   then $S^U(dx)=S^V(dx)$ on $\mathscr B(K)$. Therefore we can define a measure $\mu_X:=S^U(dx)$ on $K$ with the support on $X$. Moreover for every $J\in\mathscr B(K)$,  the function
 $$H\ni X\mapsto\mu_X(J)\in [0,1]$$ is Borel measurable.

If $\nu$ is a probability measure on $K$, then by  $\nu_\ast=\pi(\nu)$ we denote  the image probability measure on $H$ by the map $\pi$. Finally, we  define a measure $\bar\nu$ on $K$ by
$$
\bar\nu(J)=S(dx\otimes\nu)(J)=\int_H\mu_X(J)\,d\nu_\ast(X),\qquad J\in\mathscr B(K).
$$

In the following three theorems, the equation \eqref{mateer2} will be considered on $K\in\{\mathbb{S}^2,SO_3\}$. The joint proof is given after Theorem \ref{n4}.

\begin{theorem}\label{n2} Let us assume that  $[A,B]=0$ and  $B\ne 0$. Let $Z$ be a solution of the equation \eqref{mateer2} on $K\in\{\mathbb{S}^2,SO_3\}$ and let us denote by $\nu$ the law of $Z(0)$. Then the law of $Z(t)$ converges weakly as $t\to \infty$ to the measure $\bar\nu$ on $\mathscr B(K)$.
\end{theorem}

The following notation will be useful in formulation of the next result. If $\theta$ is a probability measure on $H$, then $\nu_\theta$ is a probability measure on $K$ defined via the following averaging formula.
\begin{equation}\label{eqn-nu_theta}
\nu_\theta(J)=\int_H\mu_X(J)\,d\theta(X),\qquad J\in\mathscr B(K).
\end{equation}

\begin{theorem}\label{n3} Let us assume that $[A,B]=0$ and  $B\ne 0$. Then the mapping $\theta\mapsto\nu_\theta$ is a bijection between the set of probability measures on $H$ and the set of invariant probability measures for the equation \eqref{mateer2} on $K\in\{\mathbb{S}^2,SO_3\}$.
\end{theorem}

\begin{theorem}\label{n4} Let $[A,B]=0$ and  $B\ne 0$. Then a probability measure $\nu$ on $K\in\{\mathbb{S}^2,SO_3\}$ is ergodic for the equation \eqref{mateer2} on $K$ iff there exists $X\in H$ such that $\nu=\mu_X$.
\end{theorem}

\begin{proof} Lets us observe that since $B^3=-\rho B$,
\begin{equation}\label{eqn-e^sB}
e^{sB}=\frac{1-\cos(\rho s)}{\rho^2}B^2+\frac{\sin(\rho s)}{\rho}B+I, \;\; s\in \mathbb{R}.
\end{equation}
Since $A$ and $B$ are linearly dependent, there exists $\kappa$ such that $A=\kappa B$ and, as already mentioned, every solution $Z$ of the equation \eqref{mateer2} has the form
$$
Z(t)=e^{tA+W_tB}Z(0)=e^{(\kappa t+W_t)B}Z(0)=S(\gamma(t),Z(0))
$$
where $\gamma(t)=(\cos(\kappa t+W_t),\sin(\kappa t+W_t))$. Using the Fourier series, we may easily prove that $\gamma(t)$ converges in law to the normalized uniform measure $dx$ on $\mathscr B(\Bbb S^1)$, hence $Z(t)$ converges in law to $\bar\nu=S(dx\otimes\nu)$ where $\nu$ is the law of $Z(0)$ on $\mathscr B(K)$. Consequently, Theorem \ref{n2} is proved on one hand, and on the other hand, we proved that if $\nu$ is invariant then $\nu=\bar\nu$.

However, $S(\mu\otimes\bar\nu)=\bar\nu$ for every probability measure $\mu$ on $\mathscr B(\Bbb S^1)$ (in particular, for the law of $\gamma(t)$) hence $\nu$  is invariant iff $\nu=\bar\nu$. Moreover, $S(\mu\otimes\nu_\theta)=\nu_\theta$ for every probability measure $\mu$ on $\mathscr B(\Bbb S^1)$ and every probability measure $\theta$ on $\mathscr B(H)$. Hence $\{\nu_\theta:\theta\textrm{ probability measure on }\mathscr B(H)\}$ coincides with the set of invariant probability measures for the equation \eqref{mateer2}. The injectivity of $\theta\mapsto\nu_\theta$ follows from the fact that the image measure $\pi(\nu_\theta)$ coincides with $\theta$ on $\mathscr B(H)$ and Theorem \ref{n3} is thus proved.

To prove Theorem \ref{n4}, realize the following: If a probability measure $\theta$ on $\mathscr B(H)$ is not a Dirac measure then there exists an open set $H_0\subseteq H$ such that $\theta(H_0)\in(0,1)$. If we define $H_1=H\setminus H_0$ and $\theta_i(J):=\theta(J\cap H_i)/\theta(H_i)$ for $i\in\{0,1\}$ and $J\in\mathscr B(H)$ then $\nu_\theta=\theta(H_0)\nu_{\theta_0}+\theta(H_1)\nu_{\theta_1}$ and thus $\nu_\theta$ is not extremal. On the other hand, if $\mu_X=\nu_{\delta_X}=\lambda_0\nu_{\theta_0}+\lambda_1\nu_{\theta_1}=\nu_{\lambda_0\theta_0+\lambda_1\theta_1}$ holds for some $X\in H$, $\lambda_i>0$, $\lambda_0+\lambda_1=1$ and for some probability measures $\theta_0$, $\theta_1$ on $\mathscr B(H)$ then $\delta_X=\pi(\nu_{\delta_X})=\pi(\nu_{\lambda_0\theta_0+\lambda_1\theta_1})=\lambda_0\theta_0+\lambda_1\theta_1$,
hence $\theta_0=\theta_1$ and $\mu_X$ is thus extremal.
\end{proof}

\begin{remark}\label{rem-equiv} If $[A,B]=0$, $B\ne 0$ and $\nu$ is a probability measure  on $K$ then the following three conditions are equivalent.
\begin{itemize}
\item[(i)] $\nu$ is an invariant measure for the equation \eqref{mateer2},
\item[(ii)] $S_p(\nu)=\nu$  for every $p\in\mathbb{S}^1$,
\item[(iii)] $\nu=S(dx\otimes\nu)$.
\end{itemize}
\end{remark}

\begin{remark}\label{rem-interpretation} In a finer look, Theorem \ref{n2} tells us  also what happens, as $t\to\infty$, on fibres of $\pi$. If $[A,B]=0$ then the diffusion (\ref{mateer2}) uniformizes, as $t\to\infty$, the initial distribution $\nu$ along the orbits $S^Z=\{S(p,Z):p\in\mathbb{S}^1\}$ indexed by $Z\in K$. Indeed, let $\nu^t$ be the law of the solution (with the initial law $\nu$) at time $t\ge 0$ and consider the $\nu_\ast$-unique system of conditional probabilities $(\nu^t\,[\cdot|\pi=X])_{X\in H}=(\nu^t_X)_{X\in H}$ on $\mathscr B(K)$, $\nu^t(X)=1$, to which the the measure $\nu^t$  desintegrates with respect to $\nu_\ast$. Then it can be verified by the definition of a desintegrated measure that $\nu^t_X=S(l_t\otimes\nu^0_X)$ where $l_t$ is the law of $\gamma(t)$, hence $\nu_X^t$ converges weakly to $\mu_X$ by Theorem \ref{n2}.
\end{remark}

\subsection{Numerical approximation}\label{nappr1}
The stochastic Landau-Lifshitz-Gilbert equation describes (uniform) atomistic ferromagnetic spin dynamics
at finite temperatures; it is of the form (\ref{mateer2}), with
$Az = -z \times h$ and $Bz = -z \times (h + h_{\perp})$, such that
\begin{equation}\label{eqn-explicit}
dz = - z \times h {\rm d}t -z \times (h + h_{\perp}){\rm d}W \,, \qquad \, t \geq 0\,,
\end{equation}
with $z(0) = z_0$. Here, $z_0, h \in {\mathbb R}^3$ are some unit vectors, and $h_{\perp} \in {\mathbb R}^3$ be
perpendicular to $h$. The
Hamiltonian ${\mathcal E}( \varphi) = -\langle h, \varphi\rangle_{{\mathbb R}^3}$ is conserved by the flow
for absent stochastic forcing, or $h_{\perp} = 0$.
We propose a non-dissipative, symmetric, and convergent discretization of \eqref{eqn-explicit}.
Let $I_k$ be an equi-distant mesh of size $k>0$ covering $[0,T]$. We denote
$\varphi^{n+1/2} := \frac{1}{2} ( \varphi^{n+1} + \varphi^n)$. \\

\textbf{Algorithm A.} Let $Z^0 := z_0$. For $n \in\mathbb{N}$, find $Z^{n+1} \in {\mathbb R}^3$ such that
\begin{equation}\label{sta1}
Z^{n+1} - Z^n = - k Z^{n+1/2}  \times h - Z^{n+1/2} \times (h + h_{\perp}) \Delta W_{n+1}\, .
\end{equation}
where $\Delta W_{n+1} := W(t_{n+1}) - W(t_{n}) \sim {\mathscr N}(0,k)$. \\

Let $k<1$. Similar to \cite{bbp1}, the ${\mathbb R}^3$-valued random variables $\{Z^{n+1}\}_{n\geq 0}$ exist,
are unique,  satisfy
$\vert Z^{n+1}\vert=1$ for all $n \geq 0$, and converge to the solution of \eqref{eqn-explicit}
 for $k \rightarrow 0$. Moreover, pathwise conservation of energy
${\mathcal E} (Z^{n+1})= {\mathcal E}(z_0)$ holds for $h_{\perp}=0$.
\begin{remark}\label{algB}
Let ${\mathscr Z}^{n+1} := {\mathbb E} Z^{n+1}$. Then for
$\overline{h} := \frac{h + h_{\perp}}{\vert h + h_{\perp}}$,
\begin{equation}\label{aa1}
\Bigl\vert {\mathscr Z}^{n+1} - {\mathscr Z}^{n} + k {\mathscr Z}^{n+1/2} \times h
+ \frac{k}{2} \vert h + h_{\perp}\vert^2 \Bigl( {\mathscr Z}^{n+1/2} -
 \langle
{\mathscr Z}^{n+1/2}, \overline{h}\rangle \overline{h}\Bigr) \Bigr\vert
\leq C k^{2}\, .
\end{equation}
Hence, the stochastic forcing exerts a damping in direction perpendicular to $h + h_{\perp}$.
Estimate (\ref{aa1}) follows as in the (more detailed)
Remark~\ref{am1} below for a second order equation with stochastic forcing. The main ideas
are to repeatedly use (\ref{sta1}), and approximation arguments that base on
$\vert Z^{n+1}\vert = 1$. As a consequence, the limiting equation reads
$$ {\mathscr Z}_t + {\mathscr Z} \times h + \frac{1}{2} \vert h + h_{\perp}\vert^2
\bigl( {\mathscr Z} - \langle {\mathscr Z}, \overline{h}\rangle \overline{h}\bigr) = 0\, .$$
\end{remark}

\subsection{Numerical experiments}\label{nexp1}
We perform computational studies of the stochastic Lan\-dau-Lifshitz-Gilbert
equation (\ref{eqn-explicit}) using {Algorithm A}.
The nonlinear system in {Algorithm A} is solved up to machine accuracy
by a fixed-point algorithm, cf. \cite{bbp1}.
By $N$ we denote the number of discrete sample paths of the Wiener process used
in the computations, and ${Z}_l^{n}$ denotes the numerical solution at time $t^n$
computed for the $l$-th sample path.

The unit sphere is divided into segments $\omega_{ij}\subset {\mathbb S}^2$
associated with points
$$
\mathbf{x}_{ij}=\left(\sin(i\pi /16)\cos(j\pi /16), \sin(i\pi /16)\sin(j\pi /16), \cos(i\pi /16)\right)\, ,
$$
$i=0,\dots,16$, $j=0,\dots,31$
such that
$\omega_{ij}=\left\{\mathbf{x}\in {\mathbb S}^2|\quad \mathbf{x}_{ij}
   = \arg \min_{\mathbf{x}_{lm}} |\mathbf{x} - \mathbf{x}_{lm}| \right\}$.
For the above partition of the sphere, at a fixed time level
$t^n$, we construct a piecewise constant empirical probability density function
$\hat{f}^n(\mathbf{x}): {\mathbb S}^2 \rightarrow R$ as
$$
\hat{f}^n(\mathbf{x})|_{\omega_{ij}}= \hat{f}^n(\mathbf{x}_{ij})
 = \frac{\#\{l|\, Z^{n}_l\in \omega_{ij}\}}{|\omega_{ij}|N}\, ,
$$
for $i=0,\dots,16$, $j=0,\dots,31$, where $\#\Omega$ denotes the cardinality of the set $\Omega$.

The presented results below were computed with $k=0.01$ for $N=20000$ sample paths.

For the first numerical experiment we consider
$z_0=(0,1/\sqrt(2),1/\sqrt(2))$,
$h=(0,0,1)$, $h_{\perp}=(0,0,0)$, which corresponds to $[A,B]=0$.
The resulting probability density for any $z_0$
is uniform on a circle with the center $(z_0,h)h$,
see Figure~\ref{fig:heqhper}. Further, $\mathbb{E}(z(t))\rightarrow (z_0,h)h$ for $t\rightarrow \infty$ and
$(z(t),h)=(z_0,h)$ for any $z_0$ (i.e. the pathwise angle between $z$ and $h$ is constant in time)
for any $t$. The computations agree with the statements of Theorems~\ref{n2}-\ref{n4}.


In the second experiment we set $z_0=(0,1/\sqrt{2},1/\sqrt{2})$,
$h=(0,0,1)$, $h_{\perp}=(0,1,0)$. In this case we have $[A,B]\neq 0$
and the initial probability density which is a Dirac measure concentrated at $z_0$
is expected to approach the uniform probability density function
of the unit sphere $f^S=(4\pi)^{-1}\approx 0.0796$.
We also expect that $\mathbb{E}(z(t))\rightarrow 0$ for $t\rightarrow \infty$. Because of the finite approximation of the problem, the uniform probability density $f^S$
is only attained approximately. Once the computed probability density becomes diffuse,
it fluctuates randomly around the uniform state $f^S$, see
Figure~\ref{fig:hperp}.
Analogically, the trajectory $\mathbb{E}(u(t))$ approaches
the center of the sphere and for long times fluctuates randomly around the center,
see Figure~\ref{fig:hperp_aver}.
The random fluctuations in the probability density function $\hat{f}^n$
can be significantly reduced
by taking the average over a sufficiently large number of time levels.
Here, we compute the time averaged probability density function
$\overline{f}$ over the last $100$ time levels, i.e.,
we take $\overline{f}(\mathbf{x}) = \frac{1}{100}\sum_{T/k-100}^{T/k} \hat{f}^n(\mathbf{x})$,
see Figure~\ref{fig:hperp_aver}.
The results in the second numerical experiment agree with the assertion of Theorem~\ref{thmabneq0}.



\section{Example II}\label{ExamII}

The geodesic equation on the sphere has the form $d\dot u=-|\dot u|^2u\,dt$. Consider this second order equation with a stochastic perturbation
\begin{equation}\label{secord}
d\dot u=-|\dot u|^2u\,dt+(u\times\dot u)\circ\,dW
\end{equation}
which is a Stratonovich SDE $dz=F(z)\,dt+G(z)\circ\,dW$ on the tangent bundle $T\mathbb{S}^2\subseteq\mathbb{R}^6$ driven by a standard one-dimensional Wiener process $W$ where $u\times\dot u\in\mathbb{R}^3$ is the outer product in $\mathbb{R}^3$ and
$$
F(p,\xi)=\left(\begin{array}{c}\xi\\-|\xi|^2p\end{array}\right),\quad G(p,\xi)=\left(\begin{array}{c}0\\p\times\xi\end{array}\right),\quad [G,F](p,\xi)=\left(\begin{array}{c}p\times\xi\\0\end{array}\right),\quad (p,\xi)\in T\mathbb{S}^2
$$
are tangent vector fields on $T\mathbb{S}^2$. Observe that $F$, $G$ and $[G,F]$ are also orthogonal tangent vector fields on the 3-dimensional submanifold
\begin{equation}\label{thesetmr}
M_r=\{(p,\xi)\in T\mathbb{S}^2:|\xi|=r\},\qquad r>0
\end{equation}
hence the hypothesis $(\bf{H})$ is verified. If we put $E_1=F/(r^2+r^4)^\frac 12$, $E_2=G/r$, $E_3=[G,F]/r$ then $\{E_1,E_2,E_3\}$ is an othonormal frame on $M_r$ and
$$
\operatorname{div}Y=\sum_{j=1}^3\langle d_{E_j}Y,E_j\rangle_{\mathbb{R}^6}=0,\qquad Y\in\{F,G,[F,G]\}
$$
where $d_XY(p)=\lim_{t\to 0}t^{-1}[Y(p+tX)-Y(p)]$ and so the hypothesis $(\bf{D})$ is holds. Next, if $[G,F]f=Gf=0$ on $M_r$ for some $f\in C^\infty(M_r)$ then $f$ is constant along any integral curve $\dot\gamma=[G,F]\gamma$ or $\dot\delta=G\delta$. Since
$$
\operatorname{Rng}\gamma=\{(p,\xi)\in M_r:p=\gamma_1(0)\}\quad\textrm{and}\quad\operatorname{Rng}\delta=\{(p,\xi)\in M_r:\xi=\delta_2(0)\},
$$
any two points in $M_r$ can be connected by a piecewise-smooth curve consisting of at most three integral curves of $[G,F]$ and $G$ hence the hypothesis $(\bf{F})$ is satisfied.

Finally, let $C_l$ denote the space spanned by polynomials on $\mathbb{R}^6$ of order smaller or equal to $l$. Then $F[C_l]\cup G[C_l]\subseteq C_l$ hence $C_l$ is invariant for the operator $\mathcal A$ defined in (\ref{generator}) and $\bigcup_{l=0}^\infty C_l$ is dense in $C(M_r)$ by the Stone-Weierstrass theorem and the hypothesis $(\bf{C})$ holds.

Let $\mu_r$ be the probability measure on $T\mathbb{S}^2$ supported on $M_r$ for $r>0$ such that the restriction of $\mu_r$ on $\mathscr B(M_r)$ coincides with the normalized Riemannian volume measure on $M_r$. If $\nu$ is a probability measure on $T\mathbb{S}^2$, we define the measures
$$
\nu_\ast(U)=\nu\,\{(p,\xi)\in T\mathbb{S}^2:|\xi|\in U\},\qquad U\in\mathscr B(0,\infty)
$$
$$
\bar\nu(A)=\nu\,(A\cap M_0)+\int_{(0,\infty)}\mu_r(A)\,d\nu_\ast,\qquad A\in\mathscr B(T\mathbb{S}^2).
$$

\begin{theorem}
Let $(u,\dot u)$ be a solution of (\ref{secord}) on $T\mathbb{S}^2$ with an initial distribution $\nu$ on $\mathscr B(T\mathbb{S}^2)$.
Then the laws of $(u,\dot u)$ converge weakly to $\bar\nu$ as $t\to\infty$. Moreover, $\nu$ is invariant for (\ref{secord}) 
iff $\nu=\bar\nu$ and $\nu$ is ergodic for (\ref{secord}) iff $\nu=\delta_{(p,0)}$ 
for some $p\in\mathbb{S}^2$ or $\nu=\mu_r$ for some $r>0$.
\end{theorem}

\begin{proof} The spaces $\{(p,0)\}_{p\in\mathbb{S}^2}$ and $M_r$ for $r>0$ are invariant for (\ref{secord}). Hence, if $f\in C(T\mathbb{S}^2)$ then $\int_{T\mathbb{S}^2}f\,dp_{(t,x)}$ converges to $f(p,0)$ if $x=(p,0)$ or it converges to $\int_{T\mathbb{S}^2}f\,d\mu_r$ if $x=(p,\xi)$ and $|\xi|=r>0$ by Theorem \ref{thm-2} since $\mu_r$ is the unique invariant and also attractive probability measure for (\ref{secord}) on $M_r$. The ergodicity follows from Remark \ref{erg} as ergodic probability measures are the extremal points of the set of all invariant probability measures (see e.g. Proposition 3.2.7 in \cite{daza}).
\end{proof}

\begin{remark}\label{erg} Invariant measures for (\ref{secord}) can be uniquely described as measures
$$
\nu_{a,b}(A)=a\,(A\cap M_0)+\int_{(0,\infty)}\mu_r(A)\,db,\qquad A\in\mathscr B(T\mathbb{S}^2)
$$
where $a$ and $b$ are finite measures on $\mathscr B(M_0)$ and on $\mathscr B(0,\infty)$ so that $a(M_0)+b(0,\infty)=1$.
\end{remark}
\subsection{Numerical approximation}\label{napprox2}
We propose a non-dissipative, symmetric discretization of (\ref{secord}) to construct
strong solutions and numerically study long-time asymptotics. Let $\{ (U^n, V^n)\}_{n}$ be
approximate iterates of $\{(u(t_n), \dot{u}(t_n))\}_n$ on an equi-distant mesh $I_k$ of size $k>0$,
covering $[0,T]$. We denote $d_t \varphi^{n+1} := \frac{1}{k} (\varphi^{n+1} - \varphi^n)$. \\

\textbf{Algorithm B.} Let $(U^0, V^0) := \bigl(u_0, \dot{u}(0)\bigr)$, and $U^{-1} := U^0 - k V^0$.
For $n \geq 0$, find $(U^{n+1}, V^{n+1}, \lambda^{n+1}) \in {\mathbb R}^{3+3+1}$,
such that for $\Delta W_{n+1} := W(t_{n+1}) - W(t_n) \sim {\mathscr N}(0,k)$ holds

\begin{eqnarray}\nonumber
V^{n+1}-V^n &=& k \frac{\lambda^{n+1}}{2} (U^{n+1} + U^{n-1}) + \frac{1}{4}(U^{n+1} + U^{n-1}) \times
(V^{n+1} + V^{n}) \Delta W_{n+1} \\ \label{forma1}
d_t U^{n+1} &=& V^{n+1} \\ \nonumber
 \lambda^{n+1} &=&
\left\{ \begin{array}{ll}
0 & \mbox{ for } \frac{1}{2}(U^{n+1} + U^{n-1}) = 0\,, \\
- \frac{(V^n, V^{n-1})}{\vert \frac{1}{2}(U^{n+1} + U^{n-1})\vert^2}
 & \mbox{ for } \frac{1}{2}(U^{n+1} + U^{n-1}) \neq 0 \mbox{ and } n \geq 1\,,  \\
 - \frac{(V^0, V^{1}) - \frac{1}{2} \vert V^0\vert^2}{\vert \frac{1}{2}(U^{n+1} + U^{n-1})\vert^2}
 & \mbox{ for } \frac{1}{2}(U^{1} + U^{-1}) \neq 0 \mbox{ and } n = 0\,.
\end{array}  \right.
\end{eqnarray}
The choice of the Lagrange multiplier $\lambda^{n+1}$ ensures that $\vert U^{n+1}\vert = 1$ for $n \geq 0$; the
case $n=0$ has to compensate for the fact that the defined $U^{-1}$ is not
necessarily of unit length.

To see the formula (\ref{forma1})$_3$ for $n \geq 1$, we multiply
(\ref{forma1}) with $\frac{1}{2}(U^{n+1} + U^{n-1})$ and use the discrete product formula
$$ (d_t \varphi^n, \psi^n) = - (\varphi^{n-1}, d_t \psi^n) + d_t (\varphi^n, \psi^n)$$
to find
$$ \frac{1}{2} (d_t V^{n+1}, U^{n+1} + U^{n-1}) = -\frac{1}{2} (V^n, V^{n+1}+V^{n-1}) +
\frac{1}{2} d_t \bigl( V^{n+1}, U^{n+1} + U^n - U^n + U^{n-1}\bigr)\, ,$$
where $(\cdot, \cdot)$ denotes the scalar product in ${\mathbb R}^3$, and $\vert \cdot \vert =
(\cdot, \cdot)^{1/2}$.
Since $(V^{n+1}, U^{n+1} + U^{n}) = 0$, we further obtain
\begin{eqnarray*}
&=& -\frac{1}{2} \Bigl( (V^n, V^{n+1} + V^{n-1}) + k d_t(V^{n+1},-V^n)\Bigr) \\
&=& -\frac{1}{2} \Bigl( (V^n, V^{n+1} + V^{n-1}) - (V^{n+1},V^n) + (V^n, V^{n-1})\Bigr) = -(V^n, V^{n-1})\, .
\end{eqnarray*}
Hence $-(V^n, V^{n-1}) = \lambda^{n+1} \vert \frac{1}{2}(U^{n+1} + U^{n-1})\vert^2$, which yields the formula for
$\lambda^{n+1}$ in (\ref{forma1}).

For $n = 0$, we conclude similarly, using $\langle U^0, V^0\rangle = 0$, and the definition of
$U^{-1}$.

\begin{theorem}\label{numscheme1}
Let $T >0$, and $k \leq k_0(U^0,V^0)$ be sufficiently small. For every $n \geq 0$, there
exist unique ${\mathbb R}^{3+3}$-valued random variables $(U^{n+1},V^{n+1})$ of
Algorithm A such that $\vert U^{n+1} \vert = 1$
for all $n \geq 0$, and
$$ E(V^{n+1}) = E(V^0) \qquad \mbox{where } E(\varphi) = \frac{1}{2} \vert \varphi\vert^2\, .$$
Moreover, $\{ (U^{n+1}, V^{n+1})\}_{n\geq 0}$ construct strong solutions of (\ref{forma1}) for
$k \rightarrow 0$, in the way specified
in the proof below.
\end{theorem}

Solvability of (\ref{forma1}) will be shown by an inductional argument that is based on Brouwer's fixed point theorem:
an auxiliary problem is introduced which excludes the case where $\frac{1}{2}(U^{n+1}+U^{n-1}) = 0$ when computing $\lambda^{n+1}$; for sufficiently small $k>0$, constructed solutions of the auxiliary problem are
in fact solutions of (\ref{forma1}).
\begin{proof}
{\em 1.~Auxiliary problem.} Fix $n \geq 1$. For every $0 < \epsilon \leq \frac{1}{4}$, define the continuous function ${\mathcal F}^{\omega}_{\epsilon}: {\mathbb R}^3 \rightarrow {\mathbb R}^3$ where
\begin{equation}\label{forma2}
{\mathcal F}^{\omega}_{\epsilon}(W) := \frac{2}{k}(W-U^n) + k \frac{(V^n, V^{n-1})}{\max\{ \vert W\vert^2, \epsilon\}} W - W \times (V^n - \frac{2}{k} U^n) \Delta W_{n+1}\, .
\end{equation}
We compute respectively,
\begin{eqnarray*}
\frac{2}{k} ( W-U^n,W) &\geq& \frac{2}{k} \bigl( \vert W\vert - \vert U^n\vert \bigr) \vert W\vert\,, \\
k \frac{(V^n, V^{n-1})}{\max\{ \vert W\vert^2,\epsilon\}} \vert W \vert^2 &\geq& - k \vert V^n\vert\, \vert V^{n-1}\vert\, .
\end{eqnarray*}
Since the stochastic term in (\ref{forma2}) vanishes after multiplication with $W$, there exists some
function $R_n >0$ such that
$$ \bigl( {\mathcal F}^{\omega}_{\epsilon}(W), W \bigr) \geq 0 \qquad \forall\, W \in
\bigl\{\varphi \in {\mathbb R}^3:\, \vert \varphi \vert \geq R_n(U^n, V^n, V^{n-1})\bigr\}\, .$$
Then, Brouwer's fixed point theorem implies the existence of $W^{\star} \equiv \frac{1}{2}(U^{n+1} + U^{n-1})$, such that
${\mathcal F}^{\omega}_{\epsilon}\bigl( \frac{1}{2}(U^{n+1} + U^{n-1})\bigr) = 0$ for all $\omega \in \Omega$.

The argument easily adopts to the case $n=0$.

{\em 2.~Solvability and energy identity.} We show that $\frac{1}{2}(U^{n+1}+U^{n-1})$ solves ${\mathcal F}^{\omega}_{0} \bigl(\frac{1}{2}(U^{n+1} +
U^{n-1})\bigr) = 0$. By induction, it suffices to verify that
\begin{eqnarray}\nonumber
\vert \frac{1}{2}(U^{n+1}+U^{n-1})\vert &=& \vert \frac{k}{2} V^{n+1} + \frac{1}{2}(U^{n} + U^{n-1})\vert
\geq \vert \frac{1}{2}(U^n + U^{n-1})\vert - \frac{k}{2} \vert V^{n+1}\vert \\ \nonumber
&\geq& \vert U^{n-1}\vert - \frac{k}{2} \bigl(\vert V^n\vert + \vert V^{n+1}\vert\bigr) \\ \label{formb1}
&\geq& 1-\frac{1}{4} - \frac{k}{2} \vert V^{n+1}\vert \stackrel{!}{>} \frac{1}{2}\, ,
\end{eqnarray}
for $k \leq k_0(U^0,V^0) < 1$ sufficiently small.

Let $n \geq 1$. For all $0 \leq \ell \leq n$, there holds $\vert U^{\ell}\vert = 1$, and
\begin{equation}\label{formb2}
E(V^{\ell}) = E(V^0) \, .
\end{equation}
Then $W^{\star} = \frac{1}{2}(U^{n+1} + U^{n-1})$ from Step 1.~solves
\begin{equation}\label{forma3}
k d_t V^{n+1} = \frac{\lambda^{n+1}_{\epsilon} k }{2} (U^{n+1} + U^{n-1}) + \frac{1}{4}(U^{n+1}+U^{n-1}) \times
(V^{n+1}+V^n) \Delta W_{n+1}\,,
\end{equation}
where
$$ \lambda^{n+1}_{\epsilon} = - \frac{(V^n,V^{n-1})}{\max\{\epsilon, \vert \frac{1}{2}(U^{n+1}+
U^{n-1})\vert^2 \}}\, .$$
Testing (\ref{forma3}) with $\frac{1}{2}(U^{n+1}-U^{n-1}) = \frac{k}{2}(V^{n+1}+V^{n})$, and using binomial
formula $\frac{1}{2}(U^{n+1}+U^{n-1}, U^{n+1}-U^{n-1}) = \frac{1}{2} \bigl(\vert U^{n+1}\vert^2-1\bigr)$,
as well as $\vert U^{n+1}\vert^2 \leq k^2 \vert V^{n+1}\vert^2 + \vert U^n\vert^2$, and the
inductive assumption $\vert U^n\vert^2 = 1$,
\begin{eqnarray}\nonumber
d_t  \vert V^{n+1}\vert^2 &\leq& \frac{\vert \lambda^{n+1}_{\epsilon}\vert}{4} k^2 \vert V^{n+1}\vert^2 \\
\label{einschu1}
&\leq& \frac{\vert V^n\vert\, \vert V^{n-1}\vert\, k^2 \vert V^{n+1}\vert^2}{4\max\{ \epsilon,
(1-\frac{1}{4}-k\vert V^{n+1}\vert)^2\}} \leq \frac{k^2}{4\epsilon}
\vert V^n\vert\, \vert V^{n-1}\vert\, \vert V^{n+1}\vert^2 \,,
\end{eqnarray}
for $\epsilon \leq \frac{1}{2}$. By a (repeated) use of the discrete version of Gronwall's inequality, there
exists a constant $C>0$ independent on time $T > 0$, such that
\begin{equation}\label{einschu2}
\vert V^{n+1}\vert^2 \leq C\, \vert V^0\vert^2\, .
\end{equation}
As a consequence, (\ref{formb1}) is valid, and hence
${\mathcal F}^{\omega}_{\epsilon}\bigl( \frac{1}{2}(U^{n+1} + U^{n-1})\bigr) = 0$ for $\epsilon = 0$;
therefore,
$U^{n+1}$ solves (\ref{forma1}), satisfies the sphere constraint, and conserves the
Hamiltonian, i.e., (\ref{formb2}) holds for $0 \leq \ell \leq n+1$.

For $n=0$, we argue correspondingly, starting in (\ref{einschu1}) with
$$d_t \vert V^1 \vert^2 \leq C \lambda_\epsilon^1 \vert U^{1} + U^{-1}\vert^{2} \leq
\frac{k \vert V^0 \vert (\vert V^0\vert + \vert V^{1}\vert )^2}{\max\{ \epsilon, \vert \frac{1}{2} (U^{1} + U^{-1})\vert\} }\, ,$$
from which we again infer (\ref{einschu2}), and (\ref{formb1}).

Uniqueness of $\bigl\{ (U^n, V^n)_n\bigr\}_n$ follows by an energy argument, using (\ref{formb1}),
(\ref{formb2}), $k \leq k_0$, and the discrete version of Gronwall's inequality.

{\em 3.~Convergence.} We rewrite (\ref{forma1}) in the form
\begin{eqnarray}\nonumber
d v &=&- \vert v\vert^2 u \, dt + u \times v \, \circ d W\,, \\ \label{formb3}
d u &=& v dt \,,\\ \nonumber
(u_0,v_0)&\in&T{\mathbb S}^2\, .
\end{eqnarray}
For ${\mathbb R}^3$-valued iterates $\{\varphi^n\}_{n \geq 0}$ on the mesh $I_k$ that covers
$[0,T]$ we define for every $t \in [t_n, t_{n+1})$ functions
\begin{eqnarray*}
{\mathscr \varphi}_k(t) &:=& \frac{t-t_n}{k} \varphi^{n+1} + \frac{t_{n+1}-t}{k} \varphi^n\,, \\
{\mathscr \varphi}_k^-(t) &:=& \varphi^n\,, \qquad \mbox{and } {\mathscr \varphi}^+_k(t) := \varphi^{n+1}\, .
\end{eqnarray*}
We now show that corresponding functions of Algorithm~A satisfy ${\mathbb P}$-almost surely
\begin{equation}\label{forma4}
({\mathscr U}^{\pm}_{k'}, {\mathscr V}_{k'}^{\pm}) \rightarrow (u,v) \qquad \mbox{in }
C \bigl([0,T]; {\mathbb R}\bigr)
\qquad (k' \rightarrow 0)\, ,
\end{equation}
where $(u,v)$ is strong solution of (\ref{forma1}).

It is because of the discrete sphere constraint and the (discrete) energy identity that sequences
$$ \bigl\{ ({\mathscr U}^{\pm}_k, {\mathscr V}^{\pm}_k)\bigr\}_k \subset C \bigl([0,T]; {\mathbb R}\bigr)$$
are uniformly bounded. Moreover, there holds for all $t \geq 0$
\begin{eqnarray}\nonumber
{\mathscr V}(t) &=& {\mathscr V}(0) + \int_0^{t^+} \frac{{\mathscr \lambda}^+}{2} [{\mathscr U}^+ + {\mathscr U}^- -
k {\mathscr V}^-]\, ds \\ \label{forma4a}
&& + \frac{1}{4} \int_0^{t^+} ({\mathscr U}^+ + {\mathscr U}^- - k {\mathscr V}^-) \times ({\mathscr V}^+ + {\mathscr V}^-)\, dW(s)\, , \\ \nonumber
{\mathscr U}(t) &=& {\mathscr U}(0) + \int_0^{t^+} {\mathscr V}(s)\, ds\, .
\end{eqnarray}
Then, by (\ref{forma4a})$_2$, (\ref{formb1}), and H\"older continuity property of $W$, sequences
$\bigl\{ ({\mathscr U}_k, {\mathscr V}_k)\bigr\}_{k}$ are equi-continuous. Hence, by Arzela-Ascoli theorem, there exist sub-sequences $\bigl\{ ({\mathscr U}_{k'}, {\mathscr V}_{k'})\bigr\}_{k'}$,
and continuous processes $({\mathscr U}, {\mathscr V})$ on $[0,T]$ such that
\begin{equation}\label{formb5}
\Vert {\mathscr U}_{k'} -u\Vert_{C([0,T]; {\mathbb R}^3)} +
                           \Vert {\mathscr V}_{k'} -v\Vert_{C([0,T]; {\mathbb R}^3)} \rightarrow 0
                           \qquad (k' \rightarrow 0) \qquad {\mathbb P}-\mbox{a.s.}
\end{equation}
We identify limits in (\ref{forma4a}). The only crucial term is the stochastic (Ito) integral term which may be
stated in the form
\begin{equation}\label{bb1}
\frac{1}{2} \int_0^{t^+} ({\mathscr U}^+ + {\mathscr U}^- - k {\mathscr V}^-) \times
({\mathscr V}^- + \frac{k}{2} \dot{\mathscr V})\, dW(s)\, .
\end{equation}
We easily find for every $t \in [0,T]$,
\begin{equation*}
\frac{1}{2} \int_0^{t^+} ({\mathscr U}^+ + {\mathscr U}^- - k {\mathscr V}^-) \times
{\mathscr V}^-\, {\rm d}W(s) \rightarrow \int_0^t u\times v\, dW(s)
\qquad (k \rightarrow 0) \qquad {\mathbb P}-\mbox{a.s.}
\end{equation*}
The remaining term in (\ref{bb1}) involves $\frac{k}{2}\dot{\mathscr V}$, which will be substituted by
identity (\ref{forma1})$_1$,
\begin{eqnarray}\nonumber
&& \frac{1}{2} \int_0^{t^+} ({\mathscr U}^+ + {\mathscr U}^- - k {\mathscr V}^-) \times \Bigl(
{\mathscr V}^- + k \frac{\lambda^+}{4} ({\mathscr U}^+ + {\mathscr U}^- - k {\mathscr V}^-) \\
\label{forma4c}
&&\qquad
+ \frac{1}{4} ({\mathscr U}^+ + {\mathscr U}^- - k {\mathscr V}^-) \times
({\mathscr V}^- + \frac{k}{2} \dot{\mathscr V}) \Delta W_{n+1}
\Bigr)\, dW(s)\, .
\end{eqnarray}
If compared to (\ref{bb1}), the critical factor $\frac{k}{2} \dot{\mathscr V}$ is now scaled by an
additional $\Delta W_{n+1}$; using again (\ref{forma1})$_1$, Ito's isometry, and the estimate
${\mathbb E} \vert \Delta W_{n+1}\vert^{2^p} \leq Ck^{2^{p-1}}$ then lead to
\begin{eqnarray}\nonumber
&&\frac{1}{8} \int_0^{t^+} ({\mathscr U}^+ + {\mathscr U}^-) \times \Bigl(
({\mathscr U}^+ + {\mathscr U}^- - k {\mathscr V}^-) \times ({\mathscr V}^-
+ \frac{k}{2} \dot{\mathscr V})
\Bigr) (W^+ - W^-)\, dW(s) \\ \label{klaa}
&&\qquad \rightarrow \frac{1}{2} \int_0^t u \times (u \times v)\, ds \qquad
(k \rightarrow 0) \qquad {\mathbb P}-\mbox{a.s.},
\end{eqnarray}
for all $t \in [0,T]$.
As a consequence, there holds
$$v(t) = v(0) + \int_0^t \vert v\vert^2 u\, ds + \int_0^t u \times v\, dW(s) +
\frac{1}{2} \int_0^t u \times (u \times v)\, ds\,,$$
where the last term is the Stratonovich correction.
The proof is complete.\end{proof}
\begin{remark}\label{am1}
1.~Let $\vert V^0\vert$ be constant, and $({\mathscr V}^{n+1}, {\mathscr U}^{n+1}) :=
({\mathbb E} V^{n+1}, {\mathbb E} U^{n+1})$. Then
\begin{equation}\label{c0}
\Bigl\vert {\mathscr V}^{n+1} - {\mathscr V}^n - k \bigl[ {\mathbb E}\vert V^{n}\vert^2
{\mathscr U}^{n+1}  - \frac{1}{2} {\mathscr V}^{n+1}\bigr] \Bigr\vert \leq C k^2\, ,
\end{equation}
i.e., the stochastic forcing term exerts a damping in direction ${\mathscr V}^{n+1}$. To show this result,
we start with
\begin{eqnarray}\nonumber
{\mathscr V}^{n+1} - {\mathscr V}^n &=& \frac{k}{2} {\mathbb E} \bigl[ \lambda^{n+1} (U^{n+1} + U^{n-1})\bigr]
+ \frac{1}{2} {\mathbb E} \bigl[ (U^{n+1} + U^{n-1}) \times V^{n+1/2} \Delta W_{n+1} \bigr] \\ \label{coa}
&=:& I + II\, .
\end{eqnarray}
We use Theorem~\ref{numscheme1}, and an approximation argument to conclude that
\begin{eqnarray*}
I &=& - \frac{k}{2} \, {\mathbb E} \bigl[ \langle V^n, V^{n-1}\rangle \Bigl( 1 - [1- \frac{1}{\vert \frac{1}{2}(U^{n+1}
+ U^{n-1})\vert^2}]\Bigr) (U^{n+1} + U^{n-1})\bigr] \\
&= & - \frac{k}{2} \, {\mathbb E} \bigl[ \langle V^n, V^{n-1}\rangle ({\mathscr U}^{n+1} + {\mathscr U}^{n-1})\bigr]
+ {\mathcal O}(k^3) \\
&=& - k\,  {\mathbb E}  \vert V^n\vert^2 {\mathscr U}^{n+1} + {\mathcal O}(k^2)\, ,
\end{eqnarray*}
thanks to the power property of expectations, and $\bigl\vert \vert \frac{1}{2}(U^{n+1} + U^{n-1})\vert^2 -1\bigr\vert \leq Ck^2$.

We use the identity $U^{n+1} = kV^{n+1} + U^n$ for the leading term in $II$, and properties of the vector product to conclude that
\begin{equation*}
II = \frac{k}{4} {\mathbb E} \bigl[ (V^{n+1}-V^n) \times V^{n}\Delta W_{n+1}\bigr] + \frac{1}{4} {\mathbb E} \bigl[ (U^{n} + U^{n-1}) \times
(V^{n+1} - V^n) \Delta W_{n+1}\bigr] =: II_a + II_b\, .
\end{equation*}
We easily verify $\vert II_a \vert \leq Ck^2$, thanks to (\ref{forma1})$_1$, and properties of iterates given
in Theorem~\ref{numscheme1}. For $II_b$, we use (\ref{forma1})$_1$ as well, and the relevant term is then
\begin{eqnarray*}
&& \frac{1}{16} {\mathbb E} \Bigl[ (U^n + U^{n-1}) \times \bigl( (U^{n+1} + U^{n-1}) \times (V^{n+1} + V^n)\bigr)\vert \Delta_{n+1}\vert^2\Bigr] \\
 &&\qquad = \frac{1}{2} {\mathbb E} \Bigl[ U^{n} \times \bigl( U^{n} \times V^{n}\bigr)\vert \Delta_{n+1}\vert^2\Bigr] + {\mathcal O}(k^2) \\
 &&\qquad = - \frac{k}{2} {\mathscr V}^{n} + {\mathcal O}(k^2)
 = - \frac{k}{2} {\mathscr V}^{n+1} + {\mathcal O}(k^2)\, ,
 \end{eqnarray*}
thanks to the power property of expectations, earlier boundedness results of iterates $\{ (U^n, V^n)\}_n$,
the fact that $\vert U^n\vert = 1$ for all $n \geq 0$, the cross product formula $a \times (b \times c) =
b \langle a, c\rangle - c \langle a, b\rangle$, and another approximation argument. This observation then
settles (\ref{c0}).

2.~Strong solutions of (\ref{secord}) satisfy $$\vert u(t)\vert = 1\,, \qquad E \bigl(v(t)\bigr) = E(v_0) \qquad \forall\, t \in [0,T]\,,$$
and are unique, due to Lipschitz continuity of coefficients in (\ref{secord}); hence, the whole sequence
$\bigl\{ ({\mathscr U}_k, {\mathscr V}_k)\bigr\}_k$ converges to $(u,v)$, for $k \rightarrow 0$.

3.~Increments of a Wiener process in Algorithm~A may be approximated by a sequence of general, not necessarily Gaussian random variables, which properly approximate higher moments of $\Delta W_{n+1}$;
martingale solutions of (\ref{secord}) may then be obtained by a more involved argumentation from
Algorithm~A as well, using theorems of Prohorov and Skorokhod; cf.~\cite{bcp1}.
\end{remark}

\subsection{Numerical experiments}\label{nelubo}
In this section we present some numerical
obtained by the {Algorithm B}
that has been applied to a slightly more general problem
than (\ref{secord})
$$
d\dot u=-|\dot u|^2u\,dt+ \sqrt{D}(u\times\dot u)\circ\,dW \, ,
$$
where $D$ is a fixed constant that controls the intensity of the noise term.
The Lagrange multiplier was computed as
\begin{equation}\label{lagr2}
\lambda^{n+1} = \frac{-(V^n,U^{n+1} + U^{n-1}) +
                \frac{1}{2k}(1-|U^{n-1}|^2)}{\left|\frac{1}{2}(U^{n+1} + U^{n-1})\right|^2}\, .
\end{equation}
The above formula is equivalent to the corresponding expression in (\ref{forma1}).
However, the present formulation (\ref{lagr2}) is slightly more convenient for numerical computations,
since it ensures that the round off errors in the constraint $|U^n|=1$ do not accumulate over time.
The solution of the nonlinear scheme (\ref{forma1}) is obtained
up to machine accuracy by a simple fixed-point algorithm, cf. \cite{bpsch}.

The probability density function $\hat{f}^n$ was constructed as in Example~I with
$N=20000$ sample paths.
For all computations in this section we take the time step size $k=0.001$
and the initial conditions $u(0) = (0,1,0)$, $\dot u(0) = (1,0,0)$.
The initial probability density function associated with the above initial conditions is
a Dirac delta function concentrated around $u(0)$.

In Figure~\ref{fig:d1_dens} we display the computed probability density $\hat{f}^n$ for $D=1$, $T=60$
at different time levels.
Initially the probability density function is advected in the direction of the initial velocity
and is simultaneously being diffused. For early times, the diffusion seems to act
predominantly in the direction perpendicular to the initial velocity.
In Figure~\ref{fig:d1_dens} we display the time averaged probability density function $\overline{f}$,
the trajectory $\mathbb{E}(u(t))$ and a zoom at $\mathbb{E}(u(t))$ near the center of the sphere.



The evolution of the probability density for $D=10$, $T=60$ is shown in Figure~\ref{fig:d10_dens}.
Similarly as in the previous experiment the probability density function diffuses and becomes
uniform for large time. Some advection
in the direction of the initial velocity can still be observed,
however, the overall process has a predominantly diffuse character. We observe that the overall
evolution damped due to the effects of the random forcing term, see Remark~\ref{am1} $(1)$
and Figure~\ref{fig:conv}.
In Figure~\ref{fig:d10_dens} we display the
time averaged probability density function, the trajectory
$\mathbb{E}(u(t))$ and a zoom at $\mathbb{E}(u(t))$ near the center.




Figure~\ref{fig:d100_d0.01} contains the computed trajectories of $\mathbb{E}(u(t))$
for $D=0.1$ and $D=100$.  The respective probability densities asymptotically converge towards the uniform distribution
for large time.


In Figure~\ref{fig:conv} we show the graphs of the time evolution of the approximate error
$\mathcal{E}_{\max}^n: t^n \rightarrow \max_{\mathbf{x}\in S^2}|f^n(\mathbf{x}) - {f}^S|$
for $D=0.01, 0.1, 1, 10, 100$.
The quantity $\mathcal{E}_{\max}^n$ serves as an approximation of
the distance from the uniform probability distribution in the $L_{\infty}$ norm.
Note that the oscillations in the error graphs are due to the approximation of the probability density.
The numerical experiments provide evidence
that the probability densities for all $D$ converges towards
the uniform probability density ${f}^{S}$ for $t\rightarrow\infty$.
The probability density evolutions for
decreasing values of $D$ have an increasingly ``advective'' character and the evolutions
for increasing values have an increasingly ``diffusive'' character.
It is also interesting to note, that the convergence
towards the uniform distribution becomes slower for both increasing and decreasing values of $D$.
\begin{figure}[htp]
\center
\caption{Convergence of the probability distribution of $u$ towards
a uniform distribution for different values of the coefficient $D$.}
\label{fig:conv}
\end{figure}

In the last experiment we study the long time behavior of the pair $(u,\dot{u})$ for $D=1$, $N=20000$. Towards this end, we introduce a partition of the manifold $M_1$ defined in \eqref{thesetmr}. First, we consider a partition of the unit sphere into segments $\omega^S_i$, $i=1,\dots,6$ associated with the points $x^S_i=(\pm1,0,0),(0,\pm1,0),(0,0,\pm1)$ in such a way that $x\in\mathbb{S}^2$ belongs to $\omega^S_i$ if and only if $|x-x^S_i|=\min_{1\le j\le 6}\,|x-x^S_j|$. Next, we denote by $T_i$ the tangent planes to points $x^S_i$. Fixing an $i\in\{1,\dots,6\}$, the orthogonal projections of vectors $\{x^S_1,\dots,x^S_6\}$ onto the tangent plane $T_i$ delimit $4$ sectors on $T_i$. We subsequently halve each sector obtaining thus $8$ equi-angular sectors $\gamma^1_i,\dots,\gamma^8_i$ in $T_i$. Now we introduce the following partition of $M_1$ into $6\times 8$ segments (see Figure~\ref{fig:part_bund}): a point $(p,\xi)\in T\mathbb{S}^2$ belongs to $M^j_i$ if $p\in\omega^S_i$ and the orthogonal projection of $\xi$ onto the tangent plane $T_i$ belongs to the sector $\gamma^j_i$. It can be verified by symmetries of this partition that the normalized surface volume of each $M^j_i$ is equal to $1/48$. For $n=60000$ (i.e., at time $t^n=60$) we have for $i=1,\dots,6$, $j=1,\dots,8$
$\#\{l|U^n_l \in \omega^S_i\}\in (3380,3260) \approx N/6 = 3333$ and
$\#\{l|(U^n_l, {V^n_l})\in M^j_i\}\in (386,455) \approx N/6/8 = 417$,
see Figure~\ref{fig:dens_ut} left and  Figure~\ref{fig:dens_ut} right, respectively.
The numerical experiments indicate that the point-wise
probability measure for $(u,\dot{u})$ converges to the invariant measure $\overline{\nu}$.
The (rescaled) approximate $L_\infty$ error $\mathcal{E}_{\max}$ for $(u,\dot{u})$ has
similar evolution as  the approximate  $L_\infty$ error for $u$. Moreover, it seems that the convergence
of the error in time is exponential, see Figure~\ref{fig:linftybundle}.



\appendix

\section{Lie algebra}

Let $U$ be an open set on a $C^\infty$-manifold.

\begin{itemize}
\item The set $\mathscr L$ of all smooth tangent vector fields on $U$  is  a vector space with the Jacobi bracket.  Any  vector subspace of $\mathscr L$ closed under the Jacobi bracket is called \textit{ a Lie algebra}.
\item If $\mathcal X$ is a set of smooth tangent vector fields on $U$, then we denote by $\mathscr L(\mathcal X)$ the smallest Lie algebra containing $\mathcal X$.
\item If $\mathcal A\subseteq\mathscr L$ and $p\in U$, then we define $\mathcal A(p)=\{A_p:A\in\mathcal A\}$.
\end{itemize}

\begin{proposition} Define $L_0=\operatorname{span}\{\mathcal X\}$ and $L_n=\operatorname{span}\{L_{n-1}\cup\{[A,B]:A,B\in L_{n-1}\}\}$. Then $\bigcup L_n=\mathscr L(\mathcal X)$.
\end{proposition}

\begin{proposition}\label{dimeq} Assume that $\mathcal X \subset \mathscr L$. Let $X_1,\dots,X_m,Y\in\mathscr L$ and let $f_i\in C^\infty(U)$. Then
$$
\mathscr L(X_1,\dots,X_m,Y)(p)=\mathscr L(X_1,\dots,X_m,Y+\sum_{j=1}^mf_jX_j)(p),\qquad p\in U.
$$
\end{proposition}

\begin{proof} Let us write $\mathcal A^1=\{X_1,\dots,X_m,Y\}$, $\mathcal A^2=\{X_1,\dots,X_m,Y+\sum_{j=1}^mf_jX_j\}$,
$$
\mathscr C^i=\left\{\sum_{k=1}^Kh_kL_k:h_k\in C^\infty(U),L_k\in\mathscr L(\mathcal A^i), K\in\mathbb{N}\right\}.
$$
Apparently, $\mathscr C^i$ is a Lie algebra for $i\in\{1,2\}$, $\mathcal A^i\subseteq\mathscr C^j$ whenever $\{i,j\}=\{1,2\}$ hence $\mathscr L(\mathcal A^i)\subseteq\mathscr C^j$ whenever $\{i,j\}=\{1,2\}$. But then
$$
\mathscr L(\mathcal A^i)(p)\subseteq\mathscr C^j(p)\subseteq\mathscr L(\mathcal A^j)(p).
$$
\end{proof}

\end{document}